\documentclass{article}
\usepackage[utf8]{inputenc}
\usepackage[all]{xy} 
\usepackage[T1]{fontenc}     
\usepackage{amssymb,latexsym, amsthm }
\usepackage{amsmath}
\newtheorem{theorem}{Theorem}

\newtheorem{proposition}{Proposition}

\newtheorem{corollary}{Corollary}
\newtheorem{example}{Example}
\usepackage{graphicx}

\title{On the Tensor product of archimedean $d$-algebras}
\author{Mohamed Amine BEN AMOR \\{\small Research Laboratory of Algebra, Topology, Arithmetic, and Order}\\{\small Department of Mathematics}\\{\small Faculty of Mathematical, Physical and Natural Sciences of Tunis}\\{\small Tunis-El Manar University, 2092-El Manar, Tunisia}\\ {} \\{\small Gosaef, Faculté des Sciences de Tunis, 2092 Manar II-Tunis, Tunisia.}}

\date{}

\begin{document}
\maketitle
\begin{abstract}
The aim of this work is to prove that the Riesz tensor product of two archimedean $d$-algebras is itself a $d$-algebra.
\end{abstract}

\section{Introduction}
Since Fremlin introduce the Riesz Tensor product of two vector lattices in \cite{FremlinTens}, several authors refined his construction, we cite namely \cite{GrobLabau, GrobLabau2, Schaefer}.

\noindent Recently, the Riesz Tensor product regain interest with the work of  Azouzi, Ben Amor and Jaber (see \cite{AzouziAmineJaber}) and separately Buskes and Wicksted (see \cite{Buskesfalg}). In fact, they proved  that the Riesz (Fremlin) tensor product of archimedean $f$-algebras is itself an $f$-algebra. The aim of this paper is to prove that the Riesz tensor product of archimedean $d$-algebras is itself an $d$-algebra.

\section{Preliminaries}
\noindent A Riesz space $A$ is called a lattice ordered algebra (briefly an $\ell$-algebra), if $A$ is an associative algebra such that the multiplication and the order structure are compatible. That is if $a$ and $b$ are positives in the $\ell$-algebra $A$, then so is $ab$.

\noindent The $\ell$-algebra $A$ is called $d$-algebra after Birkhoff and Pierce in \cite{BirkhoffPierce} if for every $a$ and $b$ in $A$ and every positive element $c$ in $ A$ we have \begin{align*}
(a\vee b)c =  ac\vee bc \\
c(a\vee b)=  ca\vee cb
\end{align*}
Which means that multiplication by a positive element is a lattice homomorphism.

\noindent The $\ell$-algebra  $A$ is called an almost-$f$-algebra, if for every $a$ and $b$ in $A$  we have \[
a\wedge b=0  \Longrightarrow  a\wedge b  =0 \] 

\noindent The $\ell$-algebra  $A$ is called $f$-algebra, if for every $a$ and $b$ in $A$ and every positive element $c$ in  $ A$ we have \[
a\wedge b=0  \Longrightarrow  ac\wedge b = ca\wedge b =0 \]
This definition goes back another time to the fundamental paper \cite{BirkhoffPierce} of Birkhoff and Pierce. It is quite obvious that all $f$-algebras are  $d$-algebras but the converse is far from being true as shows the next example (see \cite{Huijsmans}).

\begin{example}
$\mathbb{R}^2$ equipped with the following multiplication \[ \begin{pmatrix}
   a_1 \\ 
   a_2 \\ 
\end{pmatrix}\begin{pmatrix}
   b_1 \\ 
   b_2 \\ 
\end{pmatrix}= \begin{pmatrix}
   a_1 b_1 \\ 
   a_1 b_2 \\ 
\end{pmatrix}\]is an archimedean $d$-algebra but fails to be an $f$-algebra nor an almost-$f$-algebra.
\end{example}

\noindent It is clear then that $d$-algebras form a wider class of $\ell$-algebras than the class of $f$-algebras. Even $d$-algebras fail to be commutative as shows also the previous example while $f$-algebras as almost-$f$-algebras are, one can ask if other properties of $f$-algebras can be extended to $d$-algebras.

\noindent We recall that the universal completion of a Riesz space $E$ can be equipped with a product that made it an $f$-algebra. Throughout this paper $E^u$ will always denote the universal completion of the Riesz space $E$ and the multiplication on it will be denoted by juxtaposition.

For the elementary theory of vector lattices and $\ell$-algebras and for unexplained terminology we refer to \cite{Aliprantis, bernauHuisjmans, Huijsmans, Zaanen1, Zaanen2}.

\section{Disjointness preserving bilinear maps}

Let $E$, $F$ and $G$ be archimedean Riesz spaces. A bilinear map $b: E\times F \to G$ is called disjointness preserving if \[ \begin{array}{lcl}
x \perp y &\Longrightarrow& b(x,z) \perp b(y,z) \\
u \perp v &\Longrightarrow& b(t,u) \perp b(t,v) \\
\end{array}\]
for all $t\in E$ and $z\in F$.

A positive disjointness preserving bilinear map is called an $\ell$-bimorphism.

A Meyer-type theorem is still valid for order bounded disjointness preserving bilinear maps that is: 
\begin{theorem}[Kusraev and Tabuev]
Let $E$, $F$ and $G$ be archimedean Riesz spaces. Let $b: E\times F \to G$  be  an order bounded disjointness preserving bilinear map then there exist $b^+$, $b^-$ and $|b|$ $\ell$-bimorphisms such that \[\begin{array}{l}
|b(x,y)| =|b(|x|,|y|)| = ||b|(|x|,|y|)| = |b|(|x|,|y|)  \text{ for all } x \in E \text{ and } y\in F\\
b^+(x,y)=(b(x,y))^+ \text{ and } b^-(x,y)=(b(x,y))^- \text{ for all } x \in E^+\text{ and } y\in F^+
\end{array} \]
\end{theorem}

The previous theorem was proved by Kusraev and Tabuev in \cite{Kusraevtabuev2} and presented without proof in Theorem 33.2 in \cite{bubuskkus}.  The next theorem also stated and proved by Kusraev and Tabuev in \cite{Kusraevtabuev2}, we present here an alternative proof.
\begin{theorem}\label{dp}
Let $E$, $F$ and $G$ be archimedean Riesz spaces. Let $b: E\times F \to G$  be  an order bounded disjointness preserving bilinear map then there exists an (essentially unique) order bounded disjointness preserving operator $b^\otimes: E\bar{\otimes}F \to G$ such that \[b(x,y)=b^\otimes (x\otimes y)\] for all $x\in E$ and $y\in F$.

Moreover \[(b^+)^\otimes=(b^\otimes)^+,\ (b^-)^\otimes=(b^\otimes)^-\text{ and }(|b^\otimes|)=( |b|)^\otimes\] 
\end{theorem}

\begin{proof}
 Let $b: E\times F \to G$  be  an order bounded disjointness preserving bilinear map. According to theorem 3.4 in \cite{Kusraevtabuev2}, there exist three $\ell$-bimorphisms $b^+$, $b^-$ and $|b|$ such that \[b= b^+ - b^- \text{ and } |b|= b^+ + b^- \]From  Theorem in \cite{FremlinTens}, there exist three $\ell$-homomorphisms $T$, $S$ and $U$ from $E\bar{\otimes}F$ to $G$ such that \[\begin{array}{lcl}
 b^+(x,y) & = & S(x\otimes y) \\
 b^-(x,y) & = & T(x\otimes y) \\
 |b|(x,y) & = & U(x\otimes y) 
 \end{array}\]for all $x\in E$ and $y\in F$.
Thus, $S+T=U$ on $E\otimes F$. Since $E\otimes F$ is ru-complete in $E\bar{\otimes} F$ (see theorem in \cite{FremlinTens}) then for any $w\in E \bar{\otimes} F$, there is some $x\in E^+$ and $y\in F^+$ such that for all $\varepsilon >0$, there exists an $v_\varepsilon \in E\otimes F$  such that \[|w-v_\varepsilon | \leq \varepsilon x \otimes y\] This with the fact that $U$ is an $\ell$-homomorphism lead to \[U(|w-v_\varepsilon |)=|U(z)-U(v_\varepsilon)| \leq \varepsilon U(x \otimes y)\]which implies that \[|U(w)-(S+T)(v_\varepsilon)| \leq \varepsilon U(x \otimes y)\]Hence, $(S+T)(v_\varepsilon)$ ru-converges to $U(w)$. But since $(S+T)(v_\varepsilon)$ ru-converges to $(S+T)(w)$, it follows that  $S+T=U$ on $E\bar{\otimes} F$. Using twice lemma 5.2 in \cite{Boulabrecent}, we obtain that $S-T$ is an order bounded disjointness preserving operator. Put $b^\otimes =S-T$. Then $b^\otimes(x\otimes y)=b(x,y) $ for all $x\in E$ and $y\in F$.
Moreover, $|S-T|$ coincides with $U$ on $E\otimes F$, so we can prove using the same techniques of ru-denseness that $U=|S-T|$  on $E\bar{\otimes} F$. It follows from it that $|T-S|=|T+S|$ that is $T\perp  S$. Since $b^\otimes =S-T$, then \[(b^+)^\otimes=(b^\otimes)^+,\ (b^-)^\otimes=(b^\otimes)^-\text{ and }(|b^\otimes|)=( |b|)^\otimes\] and our proof comes to an end.
\end{proof}

\section{The Riesz tensor product of d-algebras}

From now, we shall impose as a blanket assumption that all real Riesz spaces under consideration in the sequel are archimedean. 

The aim of this part is to prove that the tensor product of archimedean $d$-algebra is itself a $d$-algebra. In order to achieve this goal, we will use the next theorem due to Kusraev and Tabuev in \cite{KusraevTabuev}. It states that:
\begin{theorem}[Kusraev and Tabuev]
Every order bounded disjointness preserving bilinear map $b : E \times F \to G$ is
representable as the product
\[b(x, y) = S(x)T(y) \ (x \in E,\  y \in F),\]
where one of the operators $S : E \to G^u$ or $T : F \to G^u$ is order bounded and disjointness preserving, while the other operator is a lattice homomorphism.
\end{theorem}

We begin our study with this technical proposition
\begin{proposition}
Let $E$, $F$ and $G$ be  Risez spaces.  Let $S$ be in $\mathfrak{L}_b(E)$ and  $T$ be in $\mathfrak{L}_b(F)$,   two order bounded disjointness preserving operators then the order bounded bilinear map. Let $b: E\times F \to G$ be an $\ell$-bimorphism then  \[\begin{array}{lccl}
\mathfrak{b} : &\ E\times F & \to & G\\
& (x,y) & \mapsto & b(S(x), T(y)) \end{array}
\]is order bounded disjointness preserving bilinear map. Moreover \[\begin{array}{l}
\mathfrak{b}^+(x,y)=b(S^+(x), T^+(y))+b(S^-(x),  T^-(y))   \text{ and }\\
 \mathfrak{b}^-(x,y)=b(S^+(x),T^-(y))+b(S^-(x) , T^+(y))\end{array}\] for all $x$ in $E$ and $y$ in $F$.
\end{proposition}

\begin{proof}
 Let $y\in F$ and $bp_y (x)= b(S^+(x),T^+(y))+b(S^-(x),T^-(y))$ and $bm_y(x)=b(S^+(x),T^-(y))+b(S^-(x),T^+(y))$ for all $x$. It is clear that $bp_y$ and $bm_y$ are linear and order bounded.
Pick now $x_1$ and $x_2$ in $E$ such that $|x_1|\wedge |x_2|=0$ . Since $S^\pm(x_i) \perp S^\pm(x_j)$ for all $i$ and $j$ in $\{1, \ 2\}$, then \[b(S^\pm(x_i), T^\pm (y)) \perp b(S^\pm(x_j),T^\pm (y))\] This fact with lemma 5.2 in \cite{Boulabrecent} yield to that  $bp_y$ and $bm_y$  are disjointness preserving and $bp_y \perp bm_y$. We prove  an analogous result by fixing an $x$ in $E$ and considering  $bp_x$ and $bm_x$ which make an end to our prove.
\end{proof}

Since the Riesz tensor product induces an $\ell$-bimorphism on $E\times F$, the next corollary follows immediately from the previous lemma.

\begin{corollary}\label{CorTec}
Let $E$ and $F$  be  Risez spaces.  Let $S$ be in $\mathfrak{L}_b(E)$ and  $T$ be in $\mathfrak{L}_b(F)$,   two order bounded disjointness preserving operators then the order bounded bilinear map then  \[\begin{array}{lccl}
\mathfrak{b} : &\ E\times F & \to & E\bar{\otimes}F\\
& (x,y) & \mapsto & S(x)\otimes T(y) \end{array}
\]is order bounded disjointness preserving bilinear map. Moreover \[\begin{array}{l}
\mathfrak{b}^+(x,y)=S^+(x)\otimes T^+(y)+S^-(x)\otimes  T^-(y)   \text{ and }\\
 \mathfrak{b}^-(x,y)=S^+(x)\otimes T^-(y)+S^-(x)\otimes T^+(y)\end{array}\] for all $x$ in $E$ and $y$ in $F$.
\end{corollary}

At this point, we gathered all ingredients we need to prove the main theorem of this work.

\begin{theorem}\label{dalg}
Let $(A_1,\star)$ and $(A_2,\star)$ be two $d$-algebras, then the Riesz Tensor product $A_1\bar{\otimes}A_2$ can be equipped with a $d$-algebra product that extends the product on the algebraic tensor product $A_1\otimes A_2$.
\end{theorem}

\begin{proof}
Since products on $d$-algebras are order bounded disjointness preserving bilinear map, then, according to Corollary 3.3 in \cite{KusraevTabuev}, there exist an $\ell$-homomorphim $S_1$ and an order bounded disjointness preserving operator $T_1$ (respectively $S_2$ and $T_2$) from $A_1$ to its universally completion $A_1^u$ (respectively from $A_2$ to $A_2^u$) such that \[ x_1\star y_1 = S_1(x_1)T_1(y_1)\] for all $x_1, y_1$ in $A_1$. 
Put \[ \begin{array}{rccl}
\sigma:& A_1\times A_2&\longrightarrow & A_1^u\bar{\otimes}A_2^u\\
 & (x_1,x_2)& \longmapsto & S_1(x_1)\otimes S_2(x_2)
\end{array}\]
and
\[ \begin{array}{rccl}
\tau:& A_1\times A_2&\longrightarrow & A_1^u\bar{\otimes}A_2^u\\
 & (y_1,y_2)& \longmapsto & T_1(y_1)\otimes T_2(y_2)
\end{array}\]
From Theorem 4.2 in\cite{FremlinTens} and Theorem \ref{dp}, it follows that there exist an $\ell$-homomorphisms $\bar{\sigma}$ and an order bounded disjointness preserving bilinear map $\bar{\tau}$ such that: \[ \begin{array}{rccl}
\bar{\sigma}:& A_1\bar{\otimes} A_2&\longrightarrow & A_1^u\bar{\otimes}A_2^u\\
\bar{\tau}:& A_1\bar{\otimes} A_2&\longrightarrow & A_1^u\bar{\otimes}A_2^u
\end{array}\]and $\bar{\sigma}(x_1\otimes x_2)=\sigma(x_1,x_2)$ and $\bar{\tau}(y_1\otimes y_2)=\tau(y_1,y_2)$ for every $x_1$ and $y_1$ in $A_1$ and $x_2$ and $y_2$ in $A_2$.

According to Theorem 11 in \cite{AzouziAmineJaber},  $A_1^u\bar{\otimes}A_2^u$ is an $f$-algebra. We can then define an order bounded disjointness preserving bilinear map $\pi$ as: \[ \begin{array}{rccl}
\pi:& A_1\bar{\otimes} A_2 \times A_1\bar{\otimes} A_2&\longrightarrow & A_1^u\bar{\otimes}A_2^u\\
 & (u,v)& \longmapsto & \bar{\sigma}(u)\bar{\tau}(v)
\end{array}\]
We affirm that the range of $\pi$ is included in $A_1\bar{\otimes} A_2$. In fact, $A_1$ and $A_2$ are respectively Riesz subspace of $A_1^u$ and $A_2^u$. Then $ A_1\bar{\otimes} A_2$ is the Riesz subspace of $A_1^u\bar{\otimes}A_2^u$ generated by the algebraic tensor product  $ A_1\otimes A_2$ (Corollary 4.5 in \cite{FremlinTens}). Then, there exist $I_1$, $J_1$ and $K_1$ (and respectively $I_2$, $J_2$ and $K_2$) finite sets, such that, for every $u$ and $v$ in $A_1\bar{\otimes} A_2$ \[u = \bigvee_{i\in I_1} \bigwedge_{j\in J_1} \sum_{k\in K_1}x_1^{(i,j,k)}\otimes x_2^{(i,j,k)}\]  and \[ v = \bigvee_{i'\in I_2} \bigwedge_{j'\in J_2} \sum_{k'\in K_2}y_1^{(i',j',k')}\otimes y_2^{(i',j',k')}\]where $x_1^{(i,j,k)}$ and  $y_1^{(i',j',k')}$ are in $A_1$ and $x_2^{(i,j,k)}$ and  $y_2^{(i',j',k')}$ are in $A_2$. Theorem \ref{dp} yields to \[\bar{\sigma}(u) = \bigvee_{i\in I_1} \bigwedge_{j\in J_1} \sum_{k\in K_1} \sigma(x_1^{(i,j,k)}\otimes x_2^{(i,j,k)}) \]   \[\bar{\tau}^+(v) = \bigvee_{i'\in I_2} \bigwedge_{j'\in J_2} \sum_{k'\in K_2}\tau^-(y_1^{(i',j',k')}\otimes y_2^{(i',j',k')})\] and  \[\bar{\tau}^-(v) = \bigvee_{i'\in I_2} \bigwedge_{j'\in J_2} \sum_{k'\in K_2}\tau^+(y_1^{(i',j',k')}\otimes y_2^{(i',j',k')})\] Then $\pi(u,v)$ is composed by finite infima, superma and sums of $\sigma(x_1^{(i,j,k)}\otimes x_2^{(i,j,k)})\tau^\pm(y_1^{(i',j',k')}\otimes y_2^{(i',j',k')})$ which, according to Corollary \ref{CorTec}, are equal to\[S_1(x_1^{(i,j,k)})T_1^\pm(y_1^{(i',j',k')}) \otimes S_2(x_2^{(i,j,k)})T_2^\pm(y_2^{(i',j',k')})\] These terms are all in $A_1\otimes A_2$ which implies that $\pi(u,v)$ is in $A_1 \bar{\otimes}A_2$ for every $u$ and $v$ in $A_1 \bar{\otimes}A_2$. The  order bounded disjointness preserving bilinear map $\pi$ extends the product on the algebraic tensor product of $A_1\otimes A_2$ to $A_1\bar{\otimes} A_2$ as desired.

It remains to prove that this multiplication is associative. But since it is associative on the algebraic tensor product, a similar and straightforward calculus than the  latter one leads to the desired conclusion and our proof comes to an end.
\end{proof}

It remains open the fact if whether or not the Riesz Tensor product of almost-$f$-algebra is itself an almost-$f$-algebra.

\bibliographystyle{plain}
\bibliography{references}
\end{document}